\documentclass[letterpaper, 10 pt, conference]{ieeeconf}  % Comment this line out if you need a4paper

\IEEEoverridecommandlockouts                              % This command is only needed if
\usepackage{ntheorem}
\usepackage{graphics} % for pdf, bitmapped graphics files
\usepackage{epsfig} % for postscript graphics files
\usepackage{amsmath} % assumes amsmath package installed
\usepackage{amssymb}  % assumes amsmath package installed
\usepackage{todonotes}

\makeatletter
\def\endthebibliography{%
  \def\@noitemerr{\@latex@warning{Empty `thebibliography' environment}}%
  \endlist
}
\makeatother

\newtheorem{assumption}{Assumption}
\newtheorem{definition}{Definition}

\newtheorem{proposition}{Proposition}
\newtheorem{remark}{Remark}
\newtheorem{lemma}{Lemma}
\newtheorem{corollary}{Corollary}
\newtheorem{theorem}{Theorem}
\title{\LARGE \bf
Secondary Controller Design for the Safety of Nonlinear Systems \\ via Sum-of-Squares Programming*
}

\author{Yankai Lin, Michelle S. Chong, and Carlos Murguia% <-this % stops a space
\thanks{The authors are with Department of Mechanical Engineering, Eindhoven University of Technology, the Netherlands.
        {\tt\small \{y.lin2,m.s.t.chong,c.g.murguia\}@ tue.nl}}%
}

\begin{document}

\maketitle

\thispagestyle{empty}
\pagestyle{empty}

%%%%%%%%%%%%%%%%%%%%%%%%%%%%%%%%%%%%%%%%%%%%%%%%%%%%%%%%%%%%%%%%%%%%%%%%%%%%%%%%
\begin{abstract}

We consider the problem of ensuring the safety of nonlinear control systems under adversarial signals. Using Lyapunov based reachability analysis, we first give sufficient conditions  to assess safety, i.e., to guarantee that the states of the control system, when starting from a given initial set, always remain in a prescribed safe set. We consider polynomial systems with semi-algebraic safe sets. Using the S-procedure for polynomial functions, safety conditions can be formulated as a Sum-Of-Squares (SOS) programme, which can be solved efficiently. When safety cannot be guaranteed, we provide tools via SOS to synthesize polynomial controllers that enforce safety of the closed loop system. The theoretical results are illustrated through numerical simulations.

\end{abstract}

%%%%%%%%%%%%%%%%%%%%%%%%%%%%%%%%%%%%%%%%%%%%%%%%%%%%%%%%%%%%%%%%%%%%%%%%%%%%%%%%
\section{INTRODUCTION}

In recent years, cyber-physical systems have gained increasing attention by researchers due to its wide applications in modern industrial systems. In these systems, computation of control laws and the physical behavior are coupled via networked communications. Though more efficient operations of the systems are enabled, more technical challenges arise at the same time. One of the pertinent vulnerabilities is when the network is compromised and malicious data is injected into the system. In \cite{cardenas2008research}, many examples including the well-known StuxNet malware incident were reported. Therefore, investigation on controller design methods that ensure safety is of significant importance. We aim to address one particular instance of the safety-ensuring control problem of which is to keep the states of the control system within a prescribed set, called a \textit{safe set}, for an infinite time horizon. In an earlier work \cite{lin2022plug}, we consider adding an output feedback dynamic secondary controller to a linear system that has already been stabilized by a pre-designed primary controller. In this work, we extend the result to polynomial nonlinear systems with semi-algebraic safe sets.

There are many approaches to the safe stabilization problem in existing literature. Among them, reachability analysis is one natural way of ensuring that for states from a given initial set are steered into the desired location without entering an unsafe region. However, it is in general computationally expensive to solve these problems exactly due to the associated partial differential equations (PDEs) that need to be dealt with \cite{margellos2011hamilton,bansal2017hamilton}. Another approach that bypasses the difficulty of dealing with PDEs is via tools of set invariance \cite{blanchini1999set}. If there exists a subset of the safe set that is forward invariant, then it is guaranteed that the states of the system always remains within the safe set. Sufficient conditions for the forward invariance of autonomous systems can be given by Lyapunov-like sufficient conditions on functions called barrier certificates \cite{prajna2007framework}. These conditions are later extended in \cite{ames2016control}, where sufficient conditions on the control barrier function (CBF) are given to guarantee robust forward invariance of the safe sets for nonlinear systems driven by control inputs. Based on these conditions, a control law that guarantees safety can be synthesized by solving a quadratic program online. Tools from dissipativity theory can also be used to formulate a similar condition that verifies safety of interconnected systems \cite{coogan2014dissipativity}.

In this work, we consider systems with polynomial dynamics and take the approach of ensuring forward invariance of a given set using tools from Sum-Of-Squares (SOS) programming to address the safe control problem. A motivation to use SOS programming is that, though some progress has been made recently\cite{tan2021high,xiao2021adaptive}, the synthesis of CBF for general nonlinear systems is still a challenging problem. In the seminal work \cite{parillo2000structured}, it is shown that a SOS program is equivalent to a semidefinite program which can be solved efficiently. Hence, by restricting the class of Lyapunov-like functions to be polynomial functions, we can efficiently translate the complicated synthesis problem to a convex optimization program. Similar ideas have been successfully applied to various nonlinear control problems such as the search of polynomial Lyapunov functions to check stability \cite{jarvis2003some}.\linebreak

\vspace{-3mm}
\noindent
Our contributions are summarized below.

\noindent
\textbf{1)} We consider the setup of using limited resources (in terms of limited access to system outputs) to design a secondary controller to ensure safety of a nonlinear controlled system under resource-limited adversaries. Sufficient conditions in terms of SOS programmes are given to synthesise polynomial state feedback controllers. This generalizes our prior work \cite{lin2022plug} on linear systems to nonlinear systems with polynomial dynamics.

\noindent
\textbf{2)} We consider the case where control inputs and external signals appear in the system dynamics. Unlike the previous work \cite{jones2019using2}, where it is assumed that the disturbance has finite energy, we consider sensor and actuator attack signals that are constrained by state-dependent upper bounds. This is to capture the fact that intelligent cyber attackers, depending on their available resources \cite{teixeira2015secure}, are constrained in the class of signals they can inject to remain stealthy, such that they can continue affecting the system without being detected.
%\end{enumerate}

The rest of this paper is organized as follows. We first present the preliminaries in Section \ref{pre}. The problem formulation is given in Section \ref{secPF}. Section \ref{sec3} presents the sufficient conditions to verify the safety of nonlinear systems based on SOS programming and the S-procedure for polynomial functions. Section \ref{sec4} proposes SOS-based synthesis tools for a polynomial output feedback secondary controller to guarantee the safety of the overall system. Section \ref{sec5} illustrates the main results via numerical simulations on a polynomial system. Lastly, conclusions and future research directions are given in Section \ref{sec6}.

\section{PRELIMINARIES}\label{pre}
\subsection{Notations}
Let $\mathbb{R}=(-\infty,\infty)$, $\mathbb{R}_{\geq 0}=[0,\infty)$, $\mathbb{R}_{>0}=(0,\infty)$ and $\mathbb{R}^{n}$ denotes the $n$-dimensional Euclidean space. We use $\mathbf{0}$ to denote the zero matrix with appropriate dimensions. For a given square matrix $R$, $\text{Tr}[R]$ denotes the trace of $R$. We use $A\succ 0$ ($A\prec 0$) and $A\succeq 0$ ($A\preceq 0$) to denote the matrix $A$ is positive (negative) definite and positive (negative) semidefinite, respectively. Given a polynomial function $p(x):\mathbb{R}^n\rightarrow\mathbb{R}$, $p$ is called SOS if there exist polynomials $p_i:\mathbb{R}^n\rightarrow\mathbb{R}$ such that $p(x) =\sum_{i=1}^{k}(p_i(x))^2$. The set of SOS polynomials and set of polynomials with real coefficients in $x$ are denoted by $\Sigma[x]$ and $\mathbb{R}[x]$, respectively. A vector of dimension $n$ composed of SOS (real) polynomial functions of $x$ is denoted by $\Sigma^n[x]$ ($\mathbb{R}^n[x]$).

\subsection{Preliminaries on Polynomial Functions}
A standard SOS program is a convex optimization problem of the following form \cite{blekherman2012semidefinite}
\begin{equation}\label{pre-sos}
    \begin{split}
        \underset{m}{\min}&\ b^\top m\\
        \text{s.t.}\ &p_i(x,m)\in\Sigma[x],\ i=1,2,\ldots,n,
        \end{split}
\end{equation}
where $p_i(x,m)=c_{i0}(x)+\sum_{j=1}^{k}c_{ij}(x)m_j$, $c_{ij}(x)\in\mathbb{R}[x]$, and $b$ is a given vector. It is shown in \cite[p. 74]{blekherman2012semidefinite} that (\ref{pre-sos}) is equivalent to a semidefinite program.

A useful tool that will be used extensively in this paper is the generalization of the S-procedure \cite{boyd1994linear} to polynomial functions. This can be done via the Positivstellensatz certificates of set containment \cite{bochnak1987geometrie}.

\begin{lemma}\label{S-pro}
Given $p_0$, $p_1$, $\cdots$, $p_m\in\mathbb{R}[x]$, if there exist $\lambda_1$, $\lambda_2$, $\cdots$, $\lambda_m\in\Sigma[x]$ such that
\begin{equation*}
    p_0-\sum_{i=1}^{m}\lambda_ip_i\in\Sigma[x],
\end{equation*}
then we have
\begin{equation*}
    \overset{m}{\underset{i=1}\bigcap}\{x|p_i(x)\geq0\}\subseteq\{x|p_0(x)\geq0\}.
\end{equation*}
\end{lemma}
\begin{proof}
    See \cite[Chapter 2.2]{jarvis2003lyapunov}.
\end{proof}

\section{PROBLEM FORMULATION}\label{secPF}
We consider the setup where the plant is modelled as a nonlinear system taking the following form
\begin{equation}\label{plt}
\begin{split}
    \dot{x}_p&=f_p(x_p)+g_p(x_p)u\\
    y&=h_p(x_p),
    \end{split}
\end{equation}
where $x_p\in\mathbb{R}^{n_p}$ is the state vector of the plant, $y\in\mathbb{R}^{n_y}$ is the measured output, $u\in\mathbb{R}^{n_u}$ is the input of the plant, function $f_p: \mathbb{R}^{n_p}\rightarrow\mathbb{R}^{n_p}$ is continuous with $f_p(0)=0$ and function $g_p: \mathbb{R}^{n_p} \rightarrow \mathbb{R}^{n_u}$ is continuous. We assume that system \eqref{plt} is stabilizable, i.e., there exists a control law $u_p$ generated by the \textit{primary controller} which has already been designed to stabilize the plant \eqref{plt} and takes the following form
\begin{equation}\label{control1}
\begin{split}
    \dot{x}_c&=f_c(x_c,y+a_y)\\
    u_p&=h_c(x_c)+a_u,
    \end{split}
\end{equation}
with controller state $x_c\in\mathbb{R}^{n_c}$. The functions $f_c$ and $h_c$ are assumed to satisfy the regularity conditions such that the primary controller \eqref{control1} exists. Controller \eqref{control1} is pre-designed to stabilize (\ref{plt}) with the input signal $u_p:\mathbb{R}^{n_c} \to \mathbb{R}^{n_u}$ and is subject to potential adversarial attacks denoted by the attack vector $a=[a_u^\top, a_y^\top]^\top\in\mathbb{R}^{n_u+n_y}$, where $a_u:\mathbb{R}_{\geq 0} \to \mathbb{R}^{n_u}$ and $a_y:\mathbb{R}_{\geq 0} \to \mathbb{R}^{n_y}$ denote actuator and sensor attacks respectively. Since the primary controller \eqref{control1} is pre-designed without being aware of the attacks, the safety of the closed loop may be compromised (precise definition is given later in Definition \ref{def-safe}). Therefore, we propose introducing a \textit{secondary controller}, that runs in conjunction with the primary controller \eqref{control1}. The secondary controller takes the form of a static output feedback controller that uses measurements of a subset of the plant outputs $y$ which are either available locally or known to be safeguarded against malicious manipulation (e.g., via encryption or watermarking):
\begin{equation}\label{control2}
\begin{split}
    u_s&=h_s(x_s),
    \end{split}
\end{equation}
where $x_s=C_sy=C_sh_p(x_p)$ is a subset of the plant measurements $y$ that are available to the secondary controller \eqref{control2}, and $u_s$ is the secondary control law. The overall controller for our \textit{safe} primary-secondary control scheme is given by
\begin{equation}\label{uuuuu}
    u=u_p+E_uu_s,
\end{equation}
where $E_u$ is a selection matrix we use to denote what entries of the primary controller are affected by the secondary. In this work, we assume $C_s$ and $E_u$ are given to model the case where a fixed set of resources are locally available. How to optimally choose $C_s$ and $E_u$ is left for future work. Note that the secondary controller \eqref{control2} uses attck-free measurements only and it generates an input that will be fed back to the plant reliably. Consequently, no attack signal $a$ appears in (\ref{control2}). The setup is illustrated in Fig.~\ref{fig:secondary_safety}.

\begin{figure}[h!]
    \centering
    \includegraphics[width=7cm]{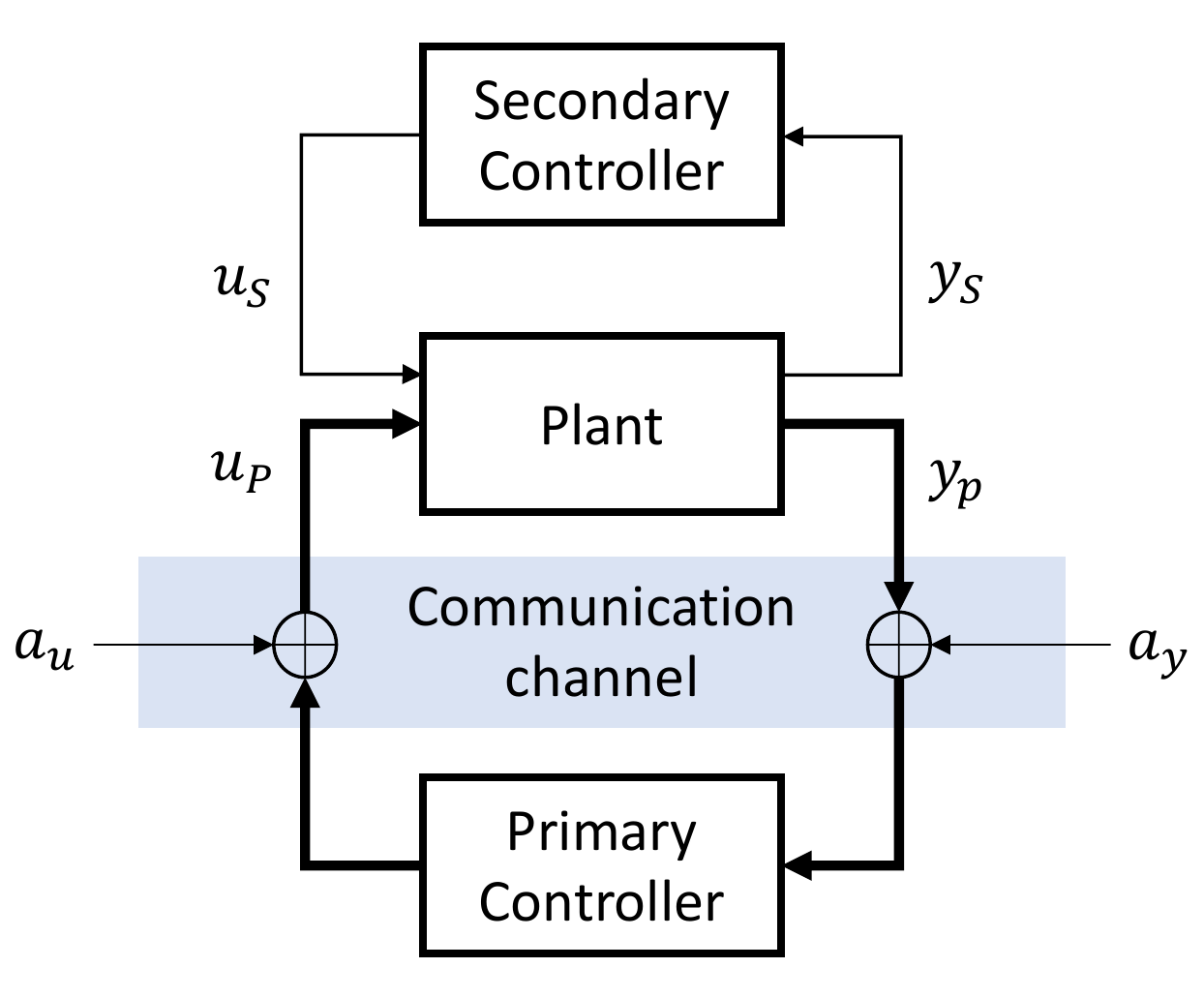}
    \caption{Ensuring safety with a secondary controller.}
    \label{fig:secondary_safety}
\end{figure}

It can be verified that the closed loop system (\ref{plt})-(\ref{uuuuu}) can be written in the following form
\begin{equation}\label{closed-loop}
\begin{split}
    \dot{x}&=f(x,a)+g(x)u_s,
    \end{split}
\end{equation}
where $x=[x_p^\top,x_c^\top]^\top$,
\begin{equation}\label{f-exp}
\begin{split}
    f(x,a)&=\left[ \begin{array}{c} f_p(x_p)+g_p(x_p)(h_c(x_c)+a_u) \\
    f_c(x_c,h_p(x_p)+a_y)\end{array} \right]\\
    g(x)&=\left[ \begin{array}{c} g_p(x_p)E_u \\
    \mathbf{0}\end{array} \right].
    \end{split}
\end{equation}

It is worth noting that the expression of the closed loop system (\ref{closed-loop}) is also able to capture the case where state feedback is used for the primary controller. Suppose we have $u_p=h_c(x_p+a_y)+a_u$, then we have
\begin{equation}\label{f-exp2}
	\begin{split}
		f(x,a)&= f_p(x_p)+g_p(x_p)(h_c(x_p+a_y)+a_u)\\
		g(x)&=g_p(x_p)E_u.
	\end{split}
\end{equation}
Throughout the paper, we will make the following assumption about the vector field of the closed loop system (\ref{closed-loop}).
\begin{assumption}\label{A-polyn}
	The closed loop system (\ref{plt})-(\ref{uuuuu}) written compactly in \eqref{closed-loop} is such that $f(x,a)\in\mathbb{R}^{n_p+n_c}[x,a]$ and $g(x)\in\mathbb{R}^{(n_p+n_c)\times n_s}[x]$.  \hfill
\end{assumption}
\begin{remark}
	By imposing Assumption \ref{A-polyn} on the closed-loop system (\ref{closed-loop}), we make the analysis more computationally tractable via SOS tools. This assumption can be satisfied if functions $f_p,g_p,h_p,f_c,h_c$ are all polynomials of their respective arguments, which may come from least square regression or polynomial approximation of another nonlinear function.
\end{remark}

The goal of the secondary controller (\ref{control2}) is to ensure that when the overall closed loop system (\ref{plt})-(\ref{uuuuu}) is subject to cyber attacks $a$, the safety of the closed loop system \eqref{closed-loop} can be ensured in the following sense.
\begin{definition}\label{def-safe}
    The closed loop system (\ref{closed-loop}) is safe if its state $x$ remains within a given safe set $\mathcal{S}$ for all $t\geq 0$.
\end{definition}

We describe the safe set by $\mathcal{S}:=\{x|s(x)\geq0\}$, where $s(x)\in\mathbb{R}[x]$. We solve the following problems in this paper.
\begin{enumerate}
    \item Give sufficient conditions to check if the primary controller \eqref{control1} alone can render the plant \eqref{plt} safe in the presence of attacks.
    \item Enhance safety of the closed loop system (\ref{closed-loop}) by synthesizing the secondary controller (\ref{control2}) such that the attacker needs to invest more resources to violate the safety condition.
\end{enumerate}

\section{INVARIANT SET BASED ANALYSIS}\label{sec3}
The first step of our analysis is to assess if the primary controller can ensure the safety of the closed loop system in the presence of the attack signal $a$. In \cite{lin2022plug}, it is assumed that the attack signal $a$ is norm bounded. This is to take into account that, intelligent adversaries often seek to remain stealthy and undetected to be able to continuously send malicious signals to the system. In this work, we impose a more general condition on the attack signal:
\begin{equation}\label{atkcons}
    a\in \mathcal{A}:=\{a|A(x,a)\geq 0\}.
\end{equation}
where $A(x,a)\in\mathbb{R}[x,a]$. The condition (\ref{atkcons}) covers the situation where adversaries that have access to the states of the system inject attack signals $a$ to the system based on their measurements of $x$. Although the requirement that $A$ is a polynomial in $x$ and $a$ might be restrictive in some cases, it generalizes the condition used in \cite{lin2022plug} which is a special case of (\ref{atkcons}) and can be used as an outer-approximation of the real constraints on the attack signal.

Since we aim to first find the worst attack signals that lead to system trajectories inside the safe set by the primary controller (\ref{control1}), we set $E_u=\mathbf{0}$ and $g(x)=\mathbf{0}$ for all $x\in\mathbb{R}^{n_p+n_c}$. Note that when $g(x)=\mathbf{0}$, the closed-loop system (\ref{closed-loop}) is a nonlinear system driven by the attack signal $a$,
\begin{equation}\label{closed-loop-pri}
    \dot{x}=f(x,a).
\end{equation}
To this end, we define the reachable set of nonlinear systems driven by external signals.
\begin{definition}
The (forward) reachable set $\mathcal{R}_a$ of the nonlinear system $\dot{x}=f(x,a)$ from the initial set $\mathcal{T}$ driven by $a\in\mathcal{A}$ is defined as the set of all trajectories $\phi(t,x(0),a)$ for all $t\geq 0$ and $x(0)\in\mathcal{T}$ where $\phi(t,x(0),a)$ is a solution to $\dot{x}=f(x,a)$ at time $t$ with the initial condition $x(0)$.
\end{definition}

If we can guarantee that the reachable set of (\ref{closed-loop-pri}) is a subset of the safe set $\mathcal{S}$, then safety of the system can be ensured since system states can only reach a set that is fully contained in the prescribed safe set. Exact computation of the reachable set of a nonlinear system can be difficult in general. In this work, we construct the set $\mathcal{E}_{a}$ as an outer approximation of the reachable set such that
\begin{equation}\label{contain}
    \mathcal{R}_a\subseteq \mathcal{E}_a,
\end{equation}
where $\mathcal{R}_a$ is the reachable set of (\ref{closed-loop-pri}) from the initial set $\mathcal{T}$. If we manage to find such an $\mathcal{E}_{a}$ such that $\mathcal{E}_{a}\subseteq\mathcal{S}$, then it is sufficient to conclude that $\mathcal{R}_a\subseteq \mathcal{S}$. We will make use of the following result which comes from Nagumo’s theorem for autonomous systems and its extension to systems with exogenous inputs by Aubin \cite{aubin2011viability}, see \cite[Section 3.1]{blanchini1999set}.

\begin{proposition}\label{Prop1}
    Given system (\ref{closed-loop-pri}), if there exists a continuously differentiable function $V:\mathbb{R}^{n_p+n_c}\rightarrow\mathbb{R}$ such that $\mathcal{T}\subseteq\mathcal{E}_{a}:=\{x\in\mathbb{R}^{n_p+n_c}|V(x)\leq1\}$ and
    \begin{equation}\label{Lya-condi}
        \frac{\partial V}{\partial x}f(x,a)\leq 0,\ \forall \hspace{1mm} V(x)=1, x\in\mathbb{R}^{n_p+n_c}, a\in\mathcal{A};
    \end{equation}
    then, we have $\mathcal{R}_a\subseteq\mathcal{E}_{a}$.
\end{proposition}

We will use the S-procedure for polynomial functions given by Lemma \ref{S-pro} to certify the set containment conditions in Proposition \ref{Prop1}. Assuming thet the initial set takes the form $\mathcal{T}:=\{x\in\mathbb{R}^{n_p+n_c}|T(x)\geq 0\}$ where $T(x)\in\mathbb{R}[x]$. Our first main result is stated below.

\begin{theorem}\label{T1}
    Consider the closed loop system (\ref{closed-loop}) with only the primary controller in feedback, i.e., $E_u=0$ and $C_s=0$. Given $(s(x),T(x))\in\mathbb{R}[x]\times \mathbb{R}[x]$ and $A(x,a)\in\mathbb{R}[x,a]$, if there exist $V(x)\in\mathbb{R}[x]$, $(\lambda_1,\lambda_2)\in\Sigma[x] \times \Sigma[x]$, $\lambda_3\in\mathbb{R}[x,a]$, and $\lambda_4\in\Sigma[x,a]$ such that
    \begin{equation}\label{sosT1}
        \begin{split}
            &s(x)-\lambda_1(1-V(x))\in\Sigma[x],\\
            &1-V(x)-\lambda_2T(x)\in\Sigma[x],\\
            &-\frac{\partial V}{\partial x}f(x,a)-\lambda_3(V(x)-1)-\lambda_4A(x,a)\in\Sigma[x,a],
        \end{split}
    \end{equation}
    then we have $\mathcal{R}_a\subseteq\mathcal{S}$.
\end{theorem}
\begin{proof}
     Applying the S-procedure for polynomial functions, the first condition in (\ref{sosT1}) guarantees that if $1-V(x)\geq0$, we have $s(x)\geq0$ which means that $\mathcal{E}_a\subseteq\mathcal{S}$. Similarly, the second condition in (\ref{sosT1}) implies $\mathcal{T}\subseteq\mathcal{E}_a$. Lastly, the third condition guarantees (\ref{Lya-condi}). By Proposition \ref{Prop1}, we have $\mathcal{R}_a\subseteq\mathcal{E}_a$ which together with $\mathcal{E}_a\subseteq\mathcal{S}$ implies $\mathcal{R}_a\subseteq\mathcal{S}$.
\end{proof}

Note that the condition (\ref{sosT1}) is a sufficient condition to guarantee that when the initial state is in $\mathcal{T}$, the state of the closed loop system never enters the unsafe region of the state space in the presence of attack signals. Any forward invariant set $\mathcal{E}_a$ that verifies (\ref{sosT1}) will suffice to guarantee safety of the system. It is also possible to optimize the safety performance of the system by minimizing an appropriately chosen objective function. One example is to minimize the volume of the forward invariant set $\mathcal{E}_a$. However, since $V$ is a polynomial function, finding the exact expression of the volume of $\mathcal{E}_a$ might be intractable. However, we can find an ellipsoid $\mathcal{E}_P:=\{x\in\mathbb{R}^{n_p+n_c}|x^\top P x\leq 1\}$, $P\succeq0$ that fully contains $\mathcal{E}_a$ and minimize the volume of the ellipsoid by minimizing the convex function $-\log \det(P)$, where $\det(P)$ is the determinant of the matrix $P$. See \cite{boyd1994linear}.

\begin{corollary}\label{C1}
    Consider the closed loop system (\ref{closed-loop}) with only the primary controller in feedback, i.e., $E_u=0$ and $C_s=0$. Given $s(x),T(x)\in\mathbb{R}[x]$ and $A(x,a)\in\mathbb{R}[x,a]$, if there exist $P\succ0$, $V(x)\in\mathbb{R}[x]$, $\lambda_1,\lambda_2,\lambda_5\in\Sigma[x]$, $\lambda_3\in\mathbb{R}[x,a]$, and $\lambda_4\in\Sigma[x,a]$ such that
    \begin{equation}\label{sosC1}
        \begin{split}
        &\underset{P,V,\lambda_{1-5}}{\min}\ -\log \det(P)\\
        \text{s.t.}\ &s(x)-\lambda_1(1-V(x))\in\Sigma[x],\\
            &1-V(x)-\lambda_2T(x)\in\Sigma[x],\\
            &-\frac{\partial V}{\partial x}f(x,a)-\lambda_3(V(x)-1)-\lambda_4A(x,a)\in\Sigma[x,a],\\
            &x^\top Px-\lambda_5(1-V(x))\in\Sigma[x],
        \end{split}
    \end{equation}
    where $\lambda_{1-5}$ denotes the set $\{\lambda_1,\lambda_2,\lambda_3,\lambda_4,\lambda_5\}$, then we have $\mathcal{R}_a\subseteq\mathcal{S}$. Moreover, $\mathcal{E}_P$ is the ellipsoid with minimal volume such that $\mathcal{E}_a\subseteq\mathcal{E}_P$.
\end{corollary}
\begin{proof}
     By the S-procedure for polynomial functions, the last condition in (\ref{sosC1}) guarantees that $\mathcal{E}_a\subseteq\mathcal{E}_P$. And the rest of proof follows from the proof of Theorem \ref{T1}.
\end{proof}
\begin{remark}
    The volume of $\mathcal{E}_P$ (or $\mathcal{E}_a$) is just one possible metric of safety. Indeed, for a $\mathcal{E}_a$ with a very small volume, it might still be the case that an element of $\mathcal{E}_a$ is very close to the boundary of the safe set $\mathcal{S}$. Interested readers are referred to \cite{jones2019using} for other possible choice of objective functions.
\end{remark}

It can be seen that the condition (\ref{sosT1}) contains bilinear SOS constraints involving decision variables $(\lambda_1,V)$ and $(\lambda_3,V)$ rendering the optimization problem non-convex. However, the constraints are linear in $\lambda_1$ and $\lambda_3$ when $V$ is fixed and linear in $V$ if $\lambda_1$ and $\lambda_3$ are fixed. As a result, in practice we can solve (\ref{sosT1}) and similarly (\ref{sosC1}) in an alternating fashion between variables $(\lambda_1,\lambda_3)$ and $V$ as done in many existing results, see \cite{jarvis2003some,yin2021backward,schweidel2022safe} for example.

In this work, we adopt an approach similar to \cite[Algorithm 2]{schweidel2022safe} by introducing a positive slack variable $\epsilon$ to the last condition in (\ref{sosT1}). The modified conditions takes the form:
\begin{equation}\label{sosT1-mod}
        \begin{split}
            &s(x)-\lambda_1(1-V(x))\in\Sigma[x],\\
            &1-V(x)-\lambda_2T(x)\in\Sigma[x],\\
            &-\frac{\partial V}{\partial x}f(x,a)+\epsilon-\lambda_3(V(x)-1)-\lambda_4A(x,a)\in\Sigma[x,a].
        \end{split}
    \end{equation}
The role of $\epsilon$ is to relax the decreasing condition on $V$ and allow $\dot{V}$ to be positive by the margin characterized by $\epsilon$. Then we alternately minimize $\epsilon$ over two bilinear groups of decision variables and repeat until $\epsilon\leq0$ is satisfied, which can be done by the following steps.
\begin{enumerate}
    \item Specify the orders of polynomials $V$ and $\lambda_{1-4}$ to be found.
    \item Start with an initial guess $\bar{V}$ and minimize $\epsilon$ over $\lambda_1$ and $\lambda_3$ subject to (\ref{sosT1-mod}).
    \item Set $\lambda_1$ and $\lambda_3$ to the values found in the previous step and minimize $\epsilon$ over $V$ subject to (\ref{sosT1-mod}).
    \item Repeat the previous two steps until an $\epsilon\leq0$ is found.
\end{enumerate}
We will refer to Steps $1)-4)$ as the alternating search algorithm.
\begin{remark}
    Given the specified orders of polynomials, the alternating search algorithm guarantees that after each step, $\epsilon$ is non-increasing. However, there is no guarantee that $\epsilon$ will decrease to be non-positive. It is worth noting that finding a $V$ and $\lambda_{1-4}$ that satisfy (\ref{sosT1}) is only a sufficient condition to guarantee that the reachable set $\mathcal{R}_{a}$ is contained within the safe set $\mathcal{S}$. If the alternating algorithm fails to give a feasible solution to (\ref{sosT1}), then one can try to increase the order of the polynomials and start the algorithm again until a maximum allowable order is reached, in which case (\ref{sosT1}) is claimed to be infeasible (though it can be possibly feasible).
\end{remark}
\begin{remark}\label{rmk-vol}
    Once a valid $V$ is found that verifies the safety of (\ref{closed-loop-pri}), this $V$ can be used as the initial value when solving (\ref{sosC1}). A similar alternating algorithm can be constructed to minimize the volume of $\mathcal{E}_P$. First, given $V$, $-\log \det(P)$ is minimized with respect to $\lambda_1$, $\lambda_3$, and $\lambda_5$. Then, the obtained $\lambda_1$, $\lambda_3$, and $\lambda_5$ are used to minimize $-\log \det(P)$ over $V$. The process is repeated until the decrease in $-\log \det(P)$ is within a specified tolerance.
\end{remark}

\section{SECONDARY CONTROLLER SYNTHESIS}\label{sec4}
When the primary controller alone is insufficient to guarantee the safety of the overall closed loop system \eqref{closed-loop}, introducing a secondary controller may achieve safety. To this end, we aim to systematically design the secondary controller in this section.  To be able to employ the SOS programming tools to computationally solve the synthesis problem, we restrict our class of secondary controller (\ref{control2}) to \textit{polynomial} static feedback, i.e., $h_s(x_s)\in\mathbb{R}[x_s]$. With the secondary controller \eqref{control2} included, the closed loop system takes the form
\begin{equation}\label{closed-loop-sec}
    \begin{split}
    \dot{x}&=f(x,a)+g(x)h_s(x_s):=\tilde{f}(x,a),
    \end{split}
\end{equation}
where the expressions of $f(x,a)$ and $g(x)$ are given in (\ref{f-exp}). Note that the new closed loop system (\ref{closed-loop-sec}) with the secondary controller included takes a form similar to (\ref{closed-loop-pri}). Therefore, we can employ Proposition \ref{Prop1} again to conclude the following result.

\begin{theorem}\label{T2}
    Consider the closed loop system (\ref{closed-loop-sec}). Given $(s(x),T(x))\in\mathbb{R}[x]\times \mathbb{R}[x]$ and $A(x,a)\in\mathbb{R}[x,a]$, if there exist $h_s(x_s)\in\mathbb{R}[x_s]$, $V(x)\in\mathbb{R}[x]$, $(\lambda_1,\lambda_2)\in\Sigma[x]\times \Sigma[x]$, $\lambda_3\in\mathbb{R}[x,a]$, and $\lambda_4\in\Sigma[x,a]$ such that
    \begin{equation}\label{sosT2}
        \begin{split}
            &s(x)-\lambda_1(1-V(x))\in\Sigma[x],\\
            &1-V(x)-\lambda_2T(x)\in\Sigma[x],\\
            &-\frac{\partial V}{\partial x}\tilde{f}(x,a)-\lambda_3(V(x)-1)-\lambda_4A(x,a)\in\Sigma[x,a],
        \end{split}
    \end{equation}
    where $\tilde{f}(x,a)=f(x,a)+g(x)h_s(x_s)$ depends on $h_s(x_s)$, then we have $\tilde{\mathcal{R}}_a\subseteq\mathcal{S}$, where $\tilde{\mathcal{R}}_a$ is the reachable set of (\ref{closed-loop-sec}) from the initial set $\mathcal{T}$.
\end{theorem}
\begin{proof}
    The proof follows from the proof of Theorem \ref{T1} by replacing $f$ by $\tilde{f}$.
\end{proof}

If condition (\ref{sosT2}) is satisfied by a set of decision variables $\{h_s,V,\lambda_1,\lambda_2,\lambda_3,\lambda_4\}$, then it is sufficient to conclude that the state of the closed loop system \eqref{closed-loop-sec} never leaves the safe set $\mathcal{S}$ when initialized in $\mathcal{T}$. In the synthesis of the secondary controller, we can also find an ellipsoidal outer-approximation of $\mathcal{E}_a$,  $\mathcal{E}_P:=\{x\in\mathbb{R}^{n_p+n_c}|x^\top P x\leq 1\}$ for some $P\succeq0$. Then we minimize the volume of the ellipsoid by minimizing the convex function $-\log \det(P)$.

\begin{corollary}\label{C2}
    Consider the closed loop system (\ref{closed-loop-sec}). Given $(s(x),T(x))\in\mathbb{R}[x]\times \mathbb{R}[x]$ and $A(x,a)\in\mathbb{R}[x,a]$, if there exist $P\succ0$, $h_s(x_s)\in\mathbb{R}[x_s]$, $V(x)\in\mathbb{R}[x]$, $(\lambda_1,\lambda_2,\lambda_5)\in\Sigma[x]\times \Sigma[x] \times \Sigma[x]$, $\lambda_3\in\mathbb{R}[x,a]$, and $\lambda_4\in\Sigma[x,a]$ such that
    \begin{equation}\label{sosC2}
        \begin{split}
        &\underset{P,h_s,V,\lambda_{1-5}}{\min}\ -\log \det(P)\\
        \text{s.t.}\ &s(x)-\lambda_1(1-V(x))\in\Sigma[x],\\
            &1-V(x)-\lambda_2T(x)\in\Sigma[x],\\
            &-\frac{\partial V}{\partial x}\tilde{f}(x,a)-\lambda_3(V(x)-1)-\lambda_4A(x,a)\in\Sigma[x,a],\\
            &x^\top Px-\lambda_5(1-V(x))\in\Sigma[x],
        \end{split}
    \end{equation}
    where $\tilde{f}(x,a)=f(x,a)+g(x)h_s(x_s)$ depends on $h_s(x_s)$, then we have $\tilde{\mathcal{R}}_a\subseteq\mathcal{S}$, where $\tilde{\mathcal{R}}_a$ is the reachable set of (\ref{closed-loop-sec}) from the initial set $\mathcal{T}$. Moreover, $\mathcal{E}_P$ is the ellipsoid with minimal volume such that $\tilde{\mathcal{R}}_a\subseteq\mathcal{E}_P$.
\end{corollary}

By introducing the secondary control term $h_s(x_s)$, a new bilinear term $\frac{\partial V}{\partial x}h_s(x_s)$ appears in (\ref{sosT2}) and (\ref{sosC2}). This, together with other bilinear terms in (\ref{sosT1}) and (\ref{sosC1}), makes \eqref{sosT2} and \eqref{sosC2}  non-convex optimization problems, respectively. Nevertheless, the constraints are linear in $\lambda_{1-5}$ and $h_s$ when $V$ is fixed and linear in $V$ if $\lambda_{1-5}$ and $h_s$ are fixed. There are no bilinear terms in (\ref{sosT2}) and (\ref{sosC2}) involving products of $\lambda_{1-5}$ and $h_s$. Thus, there is no need to perform an additional round of alternation. The variables $\lambda_{1-5}$ and $h_s$ can be solved simultaneously when $V$ is given. We again introduce the slack variable $\epsilon$ to (\ref{sosT2}) such that the modified condition takes the form
\begin{equation}\label{sosT2-mod}
        \begin{split}
            &s(x)-\lambda_1(1-V(x))\in\Sigma[x],\\
            &1-V(x)-\lambda_2T(x)\in\Sigma[x],\\
            &-\frac{\partial V}{\partial x}\tilde{f}(x,a)+\epsilon-\lambda_3(V(x)-1)-\lambda_4A(x,a)\in\Sigma[x,a].
        \end{split}
    \end{equation}
We alternately minimize $\epsilon$ over $(\lambda_{1,3},h_s)$ given $V$ and over $V$ given $(\lambda_{1,3},h_s)$ and repeat until $\epsilon\leq0$ is satisfied, which can be done by the following steps.
\begin{enumerate}
    \item Specify the orders of polynomials $V,h_s,\lambda_{1-4}$.
    \item Start with an initial value of $\bar{V}$ and minimize $\epsilon$ over $h_s$, $\lambda_1$ and $\lambda_3$ subject to (\ref{sosT2-mod}).
    \item Set $h_s$, $\lambda_1$ and $\lambda_3$ to the values found in the last step and minimize $\epsilon$ over $V$ subject to (\ref{sosT2-mod}).
    \item Repeat the previous two steps until an $\epsilon\leq0$ is found.
\end{enumerate}

\begin{remark}
    The initial guess of $V$ can be taken from the result of checking conditions (\ref{sosC1}) and (\ref{sosT1-mod}). To be more specific, if there does not exist a $V$ that satisfies (\ref{sosT1-mod}) for a non-positive $\epsilon$, then the initial value of $V$ can be set to be the one that minimizes $\epsilon$ subject to (\ref{sosT1-mod}). Hopefully, with the additional contribution by $h_s$, $\epsilon$ can be made negative after several iterations . If there does exist a $V$ that verifies the safety of the closed loop system (\ref{closed-loop-pri}), then the solution to (\ref{sosC1}) can be used. In such cases, (\ref{sosT1-mod}) will always be feasible since $h_s=0$ is a trivial secondary controller that ensures the safety of the closed loop system.
\end{remark}

\begin{remark}
    When the alternating algorithm does not find a non-positive $\epsilon$ that satisfies (\ref{sosT2-mod}), one can increase the order of the variables including $h_s$ which is the newly added term. Moreover, changing the values of $C_s$ and $E_u$ (asking for more locally available resources) might also be helpful in synthesizing a secondary controller that enforces the safety of the closed loop system.
\end{remark}

 Once valid $V$ and $h_s$ are found to satisfy (\ref{sosT2}), this $V$ and $h_s$ can be used as the initial value when solving (\ref{sosC2}) following the discussion in Remark \ref{rmk-vol}. However, it should be noted that the conditions (\ref{sosT2}) in Theorem \ref{T2}, though being sufficient to guarantee the safety of the closed loop system, are not sufficient to guarantee that the origin is asymptotic stable in the absence of the attack signal $a$. This is in contrast to the linear counterpart shown in \cite{lin2022plug}, where the linear secondary controller automatically guarantees asymptotic stability of the origin if it guarantees safety of the closed loop system for bounded attack signals. The design of a secondary controller aiming to recover the performance of the primary controlled system while ensuring safety will be left for future work.

\section{NUMERICAL SIMULATION}\label{sec5}
In this section, we illustrate our main results via numerical simulations of a second order nonlinear system. Suppose a primary controller has been designed such that the closed loop system (\ref{closed-loop-sec}) takes the following form
\begin{equation}\label{example1}
    \begin{split}
        \dot{x}_1&=-x_1+x_2+a_1\\
        \dot{x}_2&=-x_2-x_1^2x_2+a_2+h_s(x_2),
    \end{split}
\end{equation}
where $a=[a_1, a_2]^\top$ is the attack vector and measurement $x_s=x_2$ is used to design the secondary controller. For simplicity, the attack signals are assumed to satisfy $A(x,a)=A(a)=1-a_1^2+a_2^2\geq 0$. Moreover, we assume that the initial set $\mathcal{T}$ is the singleton containing the origin, i.e., $T(x)=-x_1^2-x_2^2\geq 0$. This captures the steady state for a globally asymptotically stable system when there are no attack signals. Under these conditions, we aim to keep the state $x=[x_1, x_2]^\top$ within the safe set $\mathcal{S}$ characterized by $s(x)=1.3-x_1^2-x_2^2\geq 0$.

First, we set $h_s(x_s)=0$ and test if the primary controller alone can keep the states $x$ inside the safe set $\mathcal{S}$. We alternately solve the condition (\ref{sosT1}) in Theorem 1 using SOSTOOLS \cite{papachristodoulou2013sostools} with SeDuMi being the solver \cite{doi:10.1080/10556789908805766}. In this example, we restrict our search of all polynomial variables to polynomials of orders no higher than 4. It turns out that, under these conditions, the condition (\ref{sosT1}) is infeasible. To explore the limitations of the primary controller, instead of insisting that the state stays in the safe set, we impose the condition that the state $x$ remains within the set $\{x\in\mathbb{R}^2|s(x)+\gamma\geq 0\}$ for some $\gamma>0$. We then minimize $\gamma$ subject to condition (\ref{sosT1}), alternately over $V$ and $\lambda_1, \lambda_3$. After 8 alternating iterations, $\gamma$ reaches the value of $0.19$ with $V=V_1(x)=0.67315x_1^2+0.70356x_2^2$. Thus, we have failed to find a $V$ which certifies the safety of the closed loop system via Theorem \ref{T1}. However, as discussed before, this does not mean the primarily controlled system is unsafe since there might exist polynomials of higher orders that satisfy (\ref{sosT1}). In the simulation, we have attempted to increase the order of the function $V(x)$ to 6 which, however, does not result in significant decrease of the value of $\gamma$.

Next we check if a polynomial secondary controller $h_s(x_2)$ of order no higher than 4 can be found to keep the states within the safe set. We again apply the alternating algorithm to check if condition (\ref{sosT2}) is feasible with the initial guess being $V=V_1(x)$. It turns out that, $h_s(x_2)=-0.31761x_2^3-1.2534x_2$ can ensure safety of the closed loop system (\ref{example1}) with $V=V_2(x)=0.8881x_1^2+2.669x_2^2$.

\begin{figure}[ht]

\includegraphics[width=8.5cm]{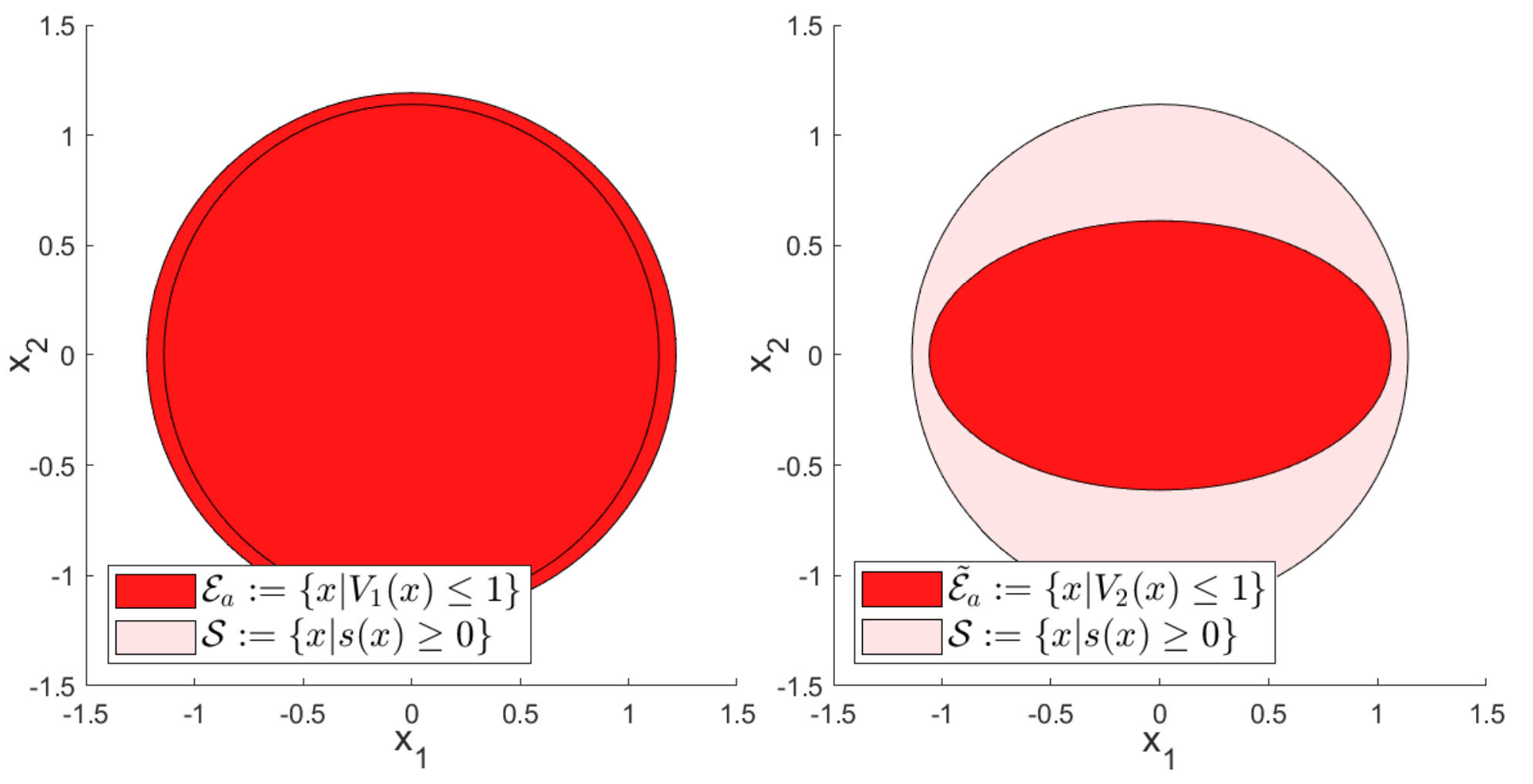}    % The printed column width is 8.4 cm.
\caption{Plots of the safe set $\mathcal{S}$ and the ellipsoidal over-approximation $\mathcal{E}_{a}$ of the reachable set $\mathcal{R}_{a}$ with the primary control only (left) and  ellipsoidal over-approximation $\tilde{\mathcal{E}}_{a}$ of the reachable set $\tilde{\mathcal{R}}_{a}$ with primary and secondary controls (right).}
\label{figsimu}

\end{figure}

The plots of the ellipsoidal over-approximations $\mathcal{E}_{a}$ and $\bar{\mathcal{E}}_{a}$ of the respective reachable sets $\mathcal{R}_a$ and $\tilde{\mathcal{R}}_{a}$ found via the corresponding alternating algorithms and the safe set are given in Fig.~\ref{figsimu}. It can be seen that, after introducing the secondary controller $h_s(x_2)=-0.31761x_2^3-1.2534x_2$ to the closed loop system, its state $x$ always remains in a subset of the safe set when initialized at the origin.


%%%%%%%%%%%%%%%%%%%%%%%%%%%%%%%%%%%%%%%%%%%%%%%%%

\section{CONCLUSIONS}\label{sec6}
In this work, based on invariant set analysis and SOS programming, we provide sufficient conditions for safety verification and control design of a class of polynomial nonlinear systems in the presence of adversarial signals. The conditions are computationally tractable via an alternating algorithm. We show that, it is possible to improve the safety performance of a nonlinear system by using a subset of sensors that are attack free. A numerical simulation on a second order nonlinear system verifies the theoretical result.

There are several possible future research direction to be explored. First, it is interesting to investigate a secondary controller design approach that recovers the performance achieved by the primary controller at least locally while ensuring safety. Another interesting topic would be the analysis of how the choice of sensors characterized by the matrix $C_s$ impacts the performance of the secondary controller.

%%%%%%%%%%%%%%%%%%%%%%%%%%%%%%%

%%%%%%%%%%%%%%%%%%%%%%%%%%%%%%%%%%%%%%%%%%%%%%%%%%%%%%%%%%%%%%%%%%%%%%%%%%%%%%%%

%%%%%%%%%%%%%%%%%%%%%%%%%%%%%%%%%%%%%%%%%%%%%%%%%%%%%%%%%%%%%%%%%%%%%%%%%%%%%%%%
%\section*{APPENDIX}

%
%%%%%%%%%%%%%%%%%%%%%%%%%%%%%%%%%%%%%%%%%%%%%%%%%%%%%%%%%%%%%%%%%%%%%%%%%%%%%%%%

\bibliographystyle{ieeetr}
\bibliography{refcdc}

%%%%%%%%%%%%%%%%%%%%%%%%%%%%%%5
\end{document}